\documentclass{amsart}
\usepackage[dvips,final]{graphics}
\usepackage{arydshln}
\usepackage[makeroom]{cancel}
\usepackage[all]{xy}
\usepackage{url}
\usepackage{multirow, blkarray}
\usepackage{booktabs}
\usepackage{textcomp}
 \usepackage[final]{epsfig}
 \usepackage{color}
\usepackage[T1]{fontenc}      
\usepackage[english,french]{babel}
\usepackage[utf8]{inputenc}
\usepackage{blindtext}

\usepackage{amsfonts, amscd, array, mathdots, epigraph, ulem}
\usepackage{amsmath}
\usepackage{amssymb}
\usepackage{amsthm}
\usepackage{mathrsfs}
\usepackage{stmaryrd}

\usepackage[pagewise]{lineno}

\begin{document}


\newtheorem{thm}{Théorème}[section]
\newtheorem{defn}[thm]{Définition}
\newtheorem{prop}[thm]{Proposition}
\newtheorem*{rem}{Remarque}
\newtheorem{ex}[thm]{Exemple}
\newtheorem{lem}[thm]{Lemme}


\title{Quelques éléments de combinatoire des matrices de $SL_{2}(\mathbb{Z})$}

\author{Flavien Mabilat}

\date{}


\address{
Flavien Mabilat,
Laboratoire de Mathématiques 
U.F.R. Sciences Exactes et Naturelles 
Moulin de la Housse - BP 1039 
51687 Reims cedex 2,
France}
\email{flavien.mabilat@univ-reims.fr}

\maketitle

\selectlanguage{english}
\begin{abstract}
A Theorem of V.Ovsienko characterizes sequences of positive integers
$(a_{1},a_{2},\ldots,a_{n})$ such that
the $(2\times2)$-matrix $\begin{pmatrix}
   a_{n} & -1 \\
    1    & 0 
   \end{pmatrix}\cdots
   \begin{pmatrix}
   a_{1} & -1 \\
    1    & 0 
    \end{pmatrix}$ is equal to $\pm Id$. In this paper, we study this equation when we replace $\pm Id$ by $\pm M$. In particular, we give a combinatorial description of the solutions of this equation in terms of dissections of convex polygons in the cases $M=\begin{pmatrix}
    0 & -1 \\
    1    & 0 
    \end{pmatrix}$ and $M=\begin{pmatrix}
    1 & 1 \\
    0 & 1 
    \end{pmatrix}$.
\end{abstract}

\selectlanguage{french}
\begin{abstract}
Un théorème de V.Ovsienko caractérise les séquences d'entiers strictement positifs
$(a_{1},\ldots,a_{n})$ vérifiant $\begin{pmatrix}
   a_{n} & -1 \\
    1    & 0 
   \end{pmatrix}\cdots
   \begin{pmatrix}
   a_{1} & -1 \\
    1    & 0 
    \end{pmatrix}=\pm Id$. Dans cette note, on étudie cette équation lorsque l'on remplace $\pm Id$ par $\pm M$. On traite en détail les cas des deux générateurs classiques du groupe modulaire $M=\begin{pmatrix}
    0 & -1 \\
    1    & 0 
    \end{pmatrix}$ et $M=\begin{pmatrix}
    1 & 1 \\
    0 & 1 
    \end{pmatrix}$, en donnant notamment une description combinatoire des solutions en terme de découpages de polygones convexes.
\\
\end{abstract}

\thispagestyle{empty}

{\bf Mots clés:} groupe modulaire; polygone; quiddité
\\
\\Declarations of interest: none
\\
\begin{flushright}
\og \textit{Quand les mystères sont très malins, ils se cachent dans la lumière.} \fg
\\ Jean Giono, \textit{Ennemonde et autres caractères}
\end{flushright}

\section{Introduction}

Le groupe modulaire 
$$SL_{2}(\mathbb{Z})=
\left\{
\begin{pmatrix}
a & b \\
c & d
   \end{pmatrix}
 \;\vert\;a,b,c,d \in \mathbb{Z},\;
 ad-bc=1
\right\}$$
et son quotient par le centre, $PSL_{2}(\mathbb{Z})=SL_{2}(\mathbb{Z})/\{\pm Id\}$,
interviennent dans de nombreux domaines mathématiques tels que la théorie des fractions continues ou la géométrie hyperbolique. Ceci explique que ce groupe ait fait l'objet de nombreux travaux et que sa structure soit maintenant bien connue. L'un des éléments les plus remarquables de cette structure est la connaissance de parties génératrices à seulement deux éléments. Il existe plusieurs façons de choisir de tels éléments et on considérera ici les deux générateurs suivants (pour une preuve de ce résultat classique on pourra consulter par exemple \cite{A}): \[T=\begin{pmatrix}
 1 & 1 \\[2pt]
    0    & 1 
   \end{pmatrix}, S=\begin{pmatrix}
   0 & -1 \\[2pt]
    1    & 0 
   \end{pmatrix}.
 \] 
\\Ce choix implique que pour tout élément $A$ de $SL_{2}(\mathbb{Z})$ il existe un entier strictement positif $m$ et des entiers relatifs $b_{1},\ldots,b_{m},d_{1},\ldots,d_{m}$ tels que \[A=T^{b_{m}}S^{d_{m}}T^{b_{m-1}}S^{d_{m-1}}\cdots T^{b_{1}}S^{d_{1}}.\] Comme $S^{2}=-Id$, il existe $\epsilon$ appartenant à $\{-1,1\}$, un entier strictement positif $l$ et des entiers relatifs $c_{1},\ldots,c_{l}$ tels que \[A= \epsilon T^{c_{l}}ST^{c_{l-1}}S\cdots T^{c_{1}}S=
\epsilon  \begin{pmatrix}
   c_{l} & -1 \\[4pt]
    1    & 0 
   \end{pmatrix}
\begin{pmatrix}
   c_{l-1} & -1 \\[4pt]
    1    & 0 
   \end{pmatrix}
   \cdots
   \begin{pmatrix}
   c_{1} & -1 \\[4pt]
    1    & 0 
    \end{pmatrix}.\]
Ici on a supposé $d_{1}$ impair. Si $d_{1}$ est pair alors on peut se ramener au cas précédent en remplaçant $T^{b_{1}}$ par $T^{b_{1}}=(-T^{2}STSTS)^{b_{1}}$, si $b_{1}$ est positif, ou, par $T^{b_{1}}=(-TSTST^{2}STS)^{\left|b_{1}\right|}$, si $b_{1}$ est négatif. On utilisera la notation $M_{n}(a_{1},\ldots,a_{n})$ pour désigner la matrice $\begin{pmatrix}
   a_{n} & -1 \\[4pt]
    1    & 0 
   \end{pmatrix}
\begin{pmatrix}
   a_{n-1} & -1 \\[4pt]
    1    & 0 
   \end{pmatrix}
   \cdots
   \begin{pmatrix}
   a_{1} & -1 \\[4pt]
    1    & 0 
    \end{pmatrix}$. De plus, on a \[S=-M_{5}(1,1,2,1,1),~T^{-1}=-M_{4}(1,2,1,1),~T=-M_{3}(1,1,2)~{\rm et}~-Id=M_{3}(1,1,1).\] Donc, toute matrice de $SL_{2}(\mathbb{Z})$ peut s'écrire sous la forme \[A=\begin{pmatrix}
   a_{n} & -1 \\[4pt]
    1    & 0 
   \end{pmatrix}
\begin{pmatrix}
   a_{n-1} & -1 \\[4pt]
    1    & 0 
   \end{pmatrix}
   \cdots
   \begin{pmatrix}
   a_{1} & -1 \\[4pt]
    1    & 0 
    \end{pmatrix}=T^{a_{n}}ST^{a_{n-1}}S\cdots T^{a_{1}}S,\] avec $a_{1},\ldots,a_{n}$ des entiers strictement positifs. Cette écriture n'est malheureusement pas unique. Par exemple, \[-Id=M_{3}(1,1,1)=M_{4}(1,2,1,2).\] Cependant, on peut essayer de trouver toutes les écritures d'une matrice sous la forme $M_{n}(a_{1},\ldots,a_{n})$. De plus, la présence d'entiers strictement positifs nous invite à chercher des descriptions combinatoires des éléments du groupe modulaire en exploitant l'idée que de tels nombres comptent des objets (en particulier géométriques). De telles descriptions permettent également de connaître plus facilement les solutions. En particulier, V.Ovsienko (voir \cite{O}) a donné une description combinatoire, en terme de découpages de polygones, des $n$-uplets d'entiers strictement positifs solutions de l'équation suivante: \begin{equation}
\tag{$E$}
M_{n}(a_1,\ldots,a_n)=\pm Id.
\end{equation} Ce résultat généralisait déjà lui-même un résultat antérieur dû à Conway et Coxeter (voir \cite{CoCo}) datant de 1973 (voir la section \ref{b} pour plus de détails). Notre objectif ici est d'étudier, pour $M \in SL_{2}(\mathbb{Z})$, la généralisation suivante de l'équation $(E)$:
\begin{equation}
\tag{$E_{M}$}
M_{n}(a_1,\ldots,a_n)=\pm M.
\end{equation} en caractérisant, si possible, les matrices $M$ pour lesquelles certaines propriétés de l'équation $(E)$ (détaillées dans les sections suivantes) sont encore vraies. On étudiera notamment en détail le cas des générateurs, c'est-à-dire les cas $M=S$ et $M=T$, en obtenant notamment une description combinatoire des solutions utilisant des découpages de polygones.

\section{Résultats principaux}\label{PDSec}
\label{a}

On s’intéresse ici, pour $M \in SL_{2}(\mathbb{Z})$, à certaines propriétés de l'équation \begin{equation}
\tag{$E_{M}$}
M_{n}(a_1,\ldots,a_n)=\pm M.
\end{equation} Notons qu'avec cette notation $(E)$ devient ($E_{Id}$). 
\\
\\On se propose ici d'étudier les cas $M=S$ et $M=T$. Pour l'équation $(E_{S})$, on introduit l'objet combinatoire suivant :

\begin{defn}
\label{22}

Soient $n \in \mathbb{N}^{*}$, $n \geq 5$ et $P$ un polygone convexe à $n+1$ sommets. Une 3$d$-dissection échancrée de $P$ est une décomposition de $P$ en sous-polygones par des diagonales ne se croisant qu'aux sommets de $P$ et vérifiant les conditions suivantes : 
\begin{itemize}
\item un seul des sous-polygones est un quadrilatère;
\item tous les autres sous-polygones intervenant dans la décomposition possèdent un nombre de sommets égal à un multiple de $3$;
\item le quadrilatère a exactement deux de ses côtés qui sont des côtés de $P$ et ces deux côtés possèdent un sommet en commun. 
\end{itemize}
\noindent On numérote ce sommet $0$ puis on numérote les autres sommets dans le sens trigonométrique. Pour chaque $i \in \llbracket 1~;~ n \rrbracket$, on note $a_{i}$ le nombre de sous-polygones possédant un nombre de sommets égal à un multiple de $3$ utilisant le sommet $i$. $(a_{1},\ldots,a_{n})$ est la quiddité de la 3$d$-dissection échancrée de $P$.

\end{defn}

Voici deux exemples de 3$d$-dissection échancrée avec leur quiddité:

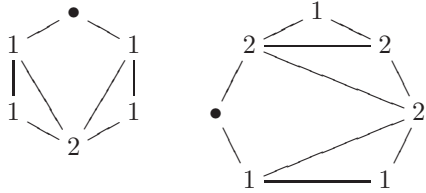
$$
\shorthandoff{; :!?}
\xymatrix @!0 @R=0.45cm @C=0.8cm
{
&\bullet\ar@{-}[ld]\ar@{-}[rd]&
\\
1\ar@{-}[dd]\ar@{-}[rddd]&&1\ar@{-}[dd]\ar@{-}[lddd]
\\
\\
1\ar@{-}[]&&1\ar@{-}[]
\\
&2\ar@{-}[lu]\ar@{-}[ru]&
}
\qquad
\xymatrix @!0 @R=0.45cm @C=0.45cm
{
&&&1\ar@{-}[rrd]\ar@{-}[lld]
\\
&2\ar@{-}[ldd]\ar@{-}[rrrrrdd]\ar@{-}[rrrr]&&&& 2\ar@{-}[rdd]&\\
\\
\bullet\ar@{-}[rdd]&&&&&& 2\ar@{-}[ldd]\\
\\
&1\ar@{-}[rrrr]\ar@{-}[rrrrruu]&&&&1
}$$

\noindent Remarquons que deux 3$d$-dissections échancrées d'un même polygone convexe peuvent avoir la même quiddité. On donne ci-dessous un exemple inspiré de la remarque 3.9 de \cite{MO} :
 
$$
\shorthandoff{; :!?}
\xymatrix @!0 @R=0.32cm @C=0.45cm
 {
&&1\ar@{-}[rrrrddddd]\ar@{-}[ld]\ar@{-}[rr]&&\bullet\ar@{-}[rd]&
\\
&2\ar@{-}[ldd]&&&& 1\ar@{-}[rdd]\ar@{-}[rdddd]\\
\\
1\ar@{-}[dd]&&&&&&1\ar@{-}[dd]\\
\\
2\ar@{-}[rdd]\ar@{-}[ruuuu]&&&&&&3\ar@{-}[ldd]\ar@{-}[lllddd]\\
\\
&1\ar@{-}[rrd]&&&& 1\ar@{-}[lld]\\
&&&2&
}
\qquad
\xymatrix @!0 @R=0.32cm @C=0.45cm
 {
&&1\ar@{-}[rrrrddddd]\ar@{-}[ld]\ar@{-}[rr]&&\bullet\ar@{-}[rd]&
\\
&2\ar@{-}[ldd]\ar@{-}[rrrrrdddd]&&&& 1\ar@{-}[rdd]\ar@{-}[rdddd]\\
\\
1\ar@{-}[dd]&&&&&&1\ar@{-}[dd]\\
\\
2\ar@{-}[rdd]\ar@{-}[rrrddd]&&&&&&3\ar@{-}[ldd]\\
\\
&1\ar@{-}[rrd]\ar@{-}&&&& 1\ar@{-}[lld]\\
&&&2&
}
\qquad
$$

On a le résultat suivant :

\begin{thm}
\label{23}

Tout $n$-uplet d'entiers strictement positifs solution de $(E_{S})$ est la quiddité d'une 3$d$-dissection échancrée d'un polygone convexe à $n+1$ sommets et réciproquement.

\end{thm}

La démonstration de ce théorème, effectuée dans la section \ref{d}, utilise en particulier le théorème d'Ovsienko rappelé dans la section suivante (voir \cite{O}).
\\
\\Pour l'équation $(E_{T})$, on introduit l'objet combinatoire ci-dessous :

\begin{defn}
\label{24}

Soient $n \in \mathbb{N}^{*}$, $n \geq 3$ et $P$ un polygone convexe à $n+2$ sommets. Une 3$d$-dissection coiffée de $P$ est une décomposition de $P$ en sous-polygones par des diagonales ne se croisant qu'aux sommets de $P$ et vérifiant les conditions suivantes : 
\begin{itemize}
\item tous les sous-polygones intervenant dans la décomposition possèdent un nombre de sommets égal à un multiple de $3$;
\item $P$ possède un triangle extérieur (c'est-à-dire que deux de ses côtés sont des côtés de $P$) auquel on affecte le poids -1;
\item tous les autres sous-polygones reçoivent le poids 1;
\item le triangle de poids -1, noté $T$, a un côté en commun avec un triangle de poids 1, noté $T'$, et, un des côtés de $T'$ est un côté de $P$.
\end{itemize}
\noindent Le sommet appartenant uniquement à $T$ est numéroté $0$ et celui appartenant uniquement à $T$ et à $T'$ est numéroté $-1$. On numérote les autres sommets de telle façon que le sommet numéroté $n$ soit adjacent au sommet numéroté $-1$. Pour chaque $i \in \llbracket 1~;~ n \rrbracket$, on note $a_{i}$ la somme des poids des sous-polygones utilisant le sommet $i$. $(a_{1},\ldots,a_{n})$ est la quiddité de la 3$d$-dissection coiffée de $P$.

\end{defn}

Voici deux exemples de 3$d$-dissection coiffée avec leur quiddité :

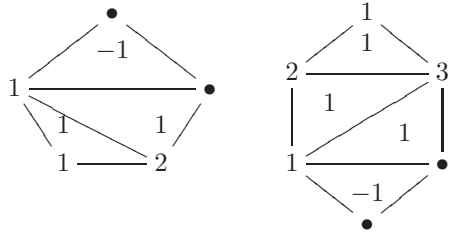
$$
\shorthandoff{; :!?}
\qquad
\xymatrix @!0 @R=0.50cm @C=0.65cm
{
&&\bullet\ar@{-}[rrdd]\ar@{-}[lldd]&
\\
&&-1
\\
1\ar@{-}[rdd]\ar@{-}[rrrdd]\ar@{-}[rrrr]&&&& \bullet\ar@{-}[ldd]\\
&1&&1&
\\
&1\ar@{-}[rr]&& 2
}
\qquad
\xymatrix @!0 @R=0.40cm @C=0.5cm
{
&&1\ar@{-}[rrdd]\ar@{-}[lldd]&
\\
&& 1
\\
2\ar@{-}[ddd]\ar@{-}[rrrr]&& && 3\ar@{-}[ddd]
\\
&1
\\
&&& 1
\\
1 \ar@{-}[rrrruuu] \ar@{-}[rrrr] &&&& \bullet
\\
&& -1
\\
&&\bullet\ar@{-}[lluu]\ar@{-}[rruu]
}
$$

\noindent Remarquons que pour obtenir une 3$d$-dissection coiffée, il suffit d'ajouter un quadrilatère coupé en deux triangles sur un côté d'une 3$d$-dissection, puis, d'ajouter les poids comme demandé dans la définition. D'autre part, on constate également que deux 3$d$-dissections coiffées d'un même polygone convexe peuvent avoir la même quiddité. On donne ci-dessous un exemple, basé lui aussi sur la remarque 3.9 de \cite{MO}. Pour éviter de surcharger la figure, les poids ne sont pas indiqués. 
 
$$
\shorthandoff{; :!?}
\xymatrix @!0 @R=0.5cm @C=0.7cm
 {
&&2\ar@{-}[rrrrddddd]\ar@{-}[ld]\ar@{-}[rr]&&2\ar@{-}[rd]\ar@{-}[rrddddd]&
\\
&2\ar@{-}[ldd]&&&& \bullet\ar@{-}[rdd]\ar@{-}[rdddd]\\
\\
1\ar@{-}[dd]&&&&&&\bullet\ar@{-}[dd]\\
\\
2\ar@{-}[rdd]\ar@{-}[ruuuu]&&&&&&3\ar@{-}[ldd]\ar@{-}[lllddd]\\
\\
&1\ar@{-}[rrd]&&&& 1\ar@{-}[lld]\\
&&&2&
}
\qquad
\xymatrix @!0 @R=0.5cm @C=0.7cm
 {
&&2\ar@{-}[rrrrddddd]\ar@{-}[ld]\ar@{-}[rr]&&2\ar@{-}[rd]\ar@{-}[rrddddd]&
\\
&2\ar@{-}[ldd]\ar@{-}[rrrrrdddd]&&&& \bullet\ar@{-}[rdd]\ar@{-}[rdddd]\\
\\
1\ar@{-}[dd]&&&&&&\bullet\ar@{-}[dd]\\
\\
2\ar@{-}[rdd]\ar@{-}[rrrddd]&&&&&&3\ar@{-}[ldd]\\
\\
&1\ar@{-}[rrd]\ar@{-}&&&& 1\ar@{-}[lld]\\
&&&2&
}
\qquad
$$

\noindent On dispose du résultat suivant :

\begin{thm}
\label{25} Soit $n \geq 3$.
\\
\\i) $(a_{1},\ldots,a_{n-1},1)$ est solution de $(E_{T})$ si et seulement si $(a_{1},\ldots,a_{n-1})$ est la quiddité d'une 3$d$-dissection échancrée d'un polygone convexe à $n$ sommets.
\\
\\ii) $(a_{1},\ldots,a_{n})$ avec $a_{n} \geq 2$ est solution de $(E_{T})$ si et seulement si $(a_{1},\ldots,a_{n})$ est la quiddité d'une 3$d$-dissection coiffée d'un polygone convexe à $n+2$ sommets.

\end{thm}

La démonstration de ce théorème, effectuée dans la section \ref{e}, utilise le théorème d'Ovsienko et le théorème \ref{23}.

\section{Théorèmes de Conway-Coxeter et Ovsienko}\label{PDSec}
\label{b}

L'objectif de cette section est de donner un bref aperçu des théorèmes de Conway-Coxeter et d'Ovsienko. Cela nous permettra d'énoncer les propriétés de l'équation $(E_{Id})$ et nous fournira des arguments qui seront réutilisés dans les sections à venir. 
\\
\\Le théorème de Conway-Coxeter relie certaines solutions de $(E_{Id})$ aux triangulations de polygones convexes (la formulation originale du résultat traitait des frises de Coxeter, voir ~\cite{CoCo,Cox} et \cite{BR} pour le lien avec $(E_{Id})$). On considère une triangulation d'un polygone convexe à $n$ sommets. Suivant \cite{CoCo}, on appelle quiddité associée à la triangulation la séquence $(a_{1},\ldots,a_{n})$
où $a_{i}$ est égal au nombre de triangles utilisant le sommet $i$. On a le résultat suivant:

\begin{thm}[Conway-Coxeter,~\cite{MO} Théorème 3.3]
\label{31}

(i) La quiddité associée à la triangulation d'un polygone convexe à $n$ sommets est un $n$-uplet d'entiers strictement positifs solution de $(E_{Id})$.

(ii) Tout $n$-uplet d'entiers strictement positifs $(a_{1},\ldots,a_{n})$ solution de $(E_{Id})$, satisfaisant la condition
$a_1+a_2+\cdots+a_n=3n-6$,
est la quiddité associée à la triangulation d'un polygone convexe à $n$ sommets.

\end{thm}

\noindent Pour une preuve de ce théorème on peut également consulter \cite{Hen}.
\\
\\Pour $n \geq 6$, il existe de nombreux $n$-uplets d'entiers strictement positifs solutions de $(E_{Id})$ qui ne peuvent pas être obtenus avec des triangulations de polygones. L'ensemble de ces $n$-uplets solutions est donné dans \cite{O}. Pour le décrire, on a besoin des deux opérations suivantes:

\begin{enumerate}
\item[(a)]
$(a_1,\ldots,a_i,a_{i+1},\ldots,a_n)
\mapsto
(a_1,\ldots,a_i+1,\,1,\,a_{i+1}+1,\ldots,a_n)$,
\\

\item[(b)]
$(a_1,\ldots,a_i,\ldots,a_n)
\mapsto
(a_1,\ldots,a_{i-1},a'_i,\,1,\,\,1,\,a''_i,a_{i+1},\ldots,a_n)$,
avec $a'_i+a''_i=a_i+1$.
\end{enumerate}

Comme indiqué précédemment, les solutions de $(E_{Id})$ sont invariantes par permutations circulaires. On considère donc les séquences $(a_1,\ldots,a_n)$ comme des séquences infinies $n$-périodiques. Les opérations ci-dessus sont ainsi possibles pour tout $i$ compris entre $1$ et $n$. 
\\
\\On montre par un simple calcul que l'opération (a) conserve la matrice $M_{n}(a_{1},\ldots,a_{n})$ et que l'opération (b) transforme la matrice $M_{n}(a_{1},\ldots,a_{n})$ en son opposée. On a également besoin de l'objet combinatoire ci-dessous définit dans \cite{O}:

\begin{defn}
\label{32}

(i) Une 3$d$-dissection est un découpage d'un polygone convexe $P$ par des diagonales ne se croisant qu'aux sommets de $P$ et tel que chaque sous-polygone résultant de ce découpage possède un nombre de sommets égal à un multiple de $3$.
\\
\\(ii) On appelle quiddité associée à la 3$d$-dissection de $P$ la séquence $(a_{1},\ldots,a_{n})$
où $a_{i}$ est égal au nombre de sous-polygones utilisant le sommet $i$.

\end{defn}

On donne ci-dessous quelques exemples de 3$d$-dissection avec leur quiddité:

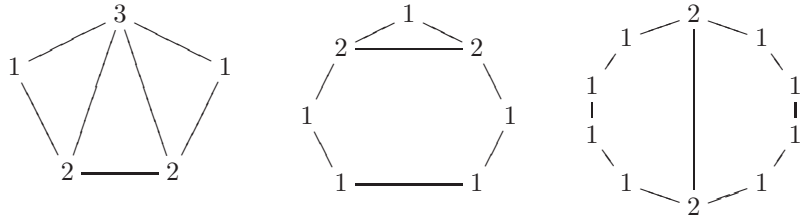
$$
\shorthandoff{; :!?}
\xymatrix @!0 @R=0.70cm @C=0.7cm
{
&&3\ar@{-}[rrd]\ar@{-}[lld]\ar@{-}[lddd]\ar@{-}[rddd]&
\\
1\ar@{-}[rdd]&&&& 1\ar@{-}[ldd]\\
\\
&2\ar@{-}[rr]&& 2
}
\qquad
\xymatrix @!0 @R=0.45cm @C=0.45cm
{
&&&1\ar@{-}[rrd]\ar@{-}[lld]
\\
&2\ar@{-}[ldd]\ar@{-}[rrrr]&&&& 2\ar@{-}[rdd]&\\
\\
1\ar@{-}[rdd]&&&&&& 1\ar@{-}[ldd]\\
\\
&1\ar@{-}[rrrr]&&&&1
}
\qquad
\xymatrix @!0 @R=0.32cm @C=0.45cm
 {
&&&2\ar@{-}[dddddddd]\ar@{-}[lld]\ar@{-}[rrd]&
\\
&1\ar@{-}[ldd]&&&& 1\ar@{-}[rdd]\\
\\
1\ar@{-}[dd]&&&&&&1\ar@{-}[dd]\\
\\
1\ar@{-}[rdd]&&&&&&1\ar@{-}[ldd]\\
\\
&1\ar@{-}[rrd]&&&& 1\ar@{-}[lld]\\
&&&2&
}
$$

Les solutions de $(E_{Id})$ sont données par les deux théorèmes de V.Ovsienko énoncés ci-dessous:

\begin{thm}[\cite{O}, Théorème 2]
\label{33}

Tout $n$-uplet d'entiers strictement positifs solution de $(E_{Id})$ peut être obtenu en appliquant les opérations (a) et (b) à $(1,1,1)$.

\end{thm}

\begin{thm}[\cite{O}, Théorème 1]
\label{34}

Un $n$-uplet d'entiers strictement positifs solution de $(E_{Id})$ est une quiddité d'une 3$d$-dissection d'un polygone convexe à $n$ sommets et réciproquement.

\end{thm}

La démonstration utilise le lemme suivant (pour une preuve complète voir \cite{O}).

\begin{lem}[\cite{O}, lemme 2.1]
\label{35}

Si $(a_{1},\ldots,a_{n})$ est une solution de $(E_{Id})$ alors il existe $i$ dans $\llbracket 1~;~ n \rrbracket$ tel que $a_{i}=1$.

\end{lem}

Ce lemme permet de montrer le théorème \ref{33} par récurrence. La démonstration du théorème \ref{34} repose sur l'interprétation géométrique des opérations (a) et (b). Si on se donne une 3$d$-dissection d'un polygone convexe à $n$ sommets, alors l'opération (a) consiste à rajouter un triangle sur l'arrête $(i,i+1)$ et l'opération (b) consiste à rajouter deux nouveaux sommets entre deux copies du sommet $i$ (voir ~\cite{O}). On en déduit le résultat par récurrence.

\begin{rem} 

{\rm 
M.Cuntz et T.Holm ont donné une description combinatoire des $n$-uplets d'entiers vérifiant l'équation $M_{n}(a_{1},\ldots,a_{n})=-Id$ en terme de triangulations pondérées (voir \cite{CH} pour plus de détails).
}

\end{rem}
 
\section{L'équation $M_{n}(a_{1},\ldots,a_{n})= \pm M$}
\label{c}

Dans cette section, on s'intéresse, pour $M \in SL_{2}(\mathbb{Z})$, aux propriétés de l'équation 
$(E_{M})$. On commence par chercher une condition suffisante sur $M$ pour que tous les $n$-uplets d'entiers strictement positifs solutions de $(E_{M})$ contiennent un 1.

\begin{prop}
\label{41}

Soit $M \in SL_{2}(\mathbb{Z})$. Si $M$ est d'ordre fini alors tous les $n$-uplets d'entiers strictement positifs solutions de $(E_{M})$ contiennent un 1.

\end{prop}

\begin{proof}

Soient $M \in SL_{2}(\mathbb{Z})$ d'ordre fini égal à $k$ et $(a_1,\ldots,a_n)$ une solution de $(E_{M})$. Il existe $\epsilon$ dans $\{ -1,1\}$ tel que $M_{n}(a_1,\ldots,a_n)=\epsilon M$. Ainsi, \[M_{kn}(a_1,\ldots,a_n,\ldots,a_1,\ldots,a_n)=M_{n}(a_1,\ldots,a_n)^{k}=\epsilon^{k} Id.\] Ainsi, $(a_1,\ldots,a_n,\ldots,a_1,\ldots,a_n)$ est solution de $(E_{Id})$ et donc, par le lemme \ref{35}, un des $a_i$ est égal à 1.

\end{proof}	

\begin{rem}

{\rm Cette proposition n'a pas de réciproque (voir section \ref{e}).}

\end{rem}	

Une des propriétés les plus intéressantes vérifiées par $(E_{Id})$ est l'invariance par permutations circulaires de ces solutions, c'est-à-dire que si $(a_{1},\ldots,a_{n})$ est solution de $(E_{Id})$ alors $(a_{2},\ldots,a_{n},a_{1})$ l'est également. En effet, si celle-ci n'est pas vérifiée, il sera plus difficile d'obtenir une description combinatoire "simple" des solutions de $(E_{M})$ en terme de découpages de polygones puisqu'il faudra prévoir une façon d'introduire un point de départ dans la lecture de la séquence associée au découpage. On va donc chercher les matrices $M$ dont les écritures sous la forme $M_{n}(a_1,\ldots,a_n)$ sont invariantes par permutations circulaires. On dira que $M$ vérifie $\mathcal{P}$ si les solutions de $M_{k}(a_{1},\ldots,a_{k})=M$ sont invariantes par permutations circulaires quelle que soit la valeur de $k$.

\begin{prop}
\label{21}

$Id$ et $-Id$ sont les seules matrices vérifiant $\mathcal{P}$.

\end{prop}

\begin{proof}

Soient $n \in \mathbb{N}^{*}$, $(a_{1},\ldots,a_{n})$ un $n$-uplet d'entiers strictement positifs et $k \in  \llbracket 1~;~ n-1 \rrbracket$. On a : 

\begin{eqnarray*}
M_{n}(a_{k+1},\ldots,a_{n},a_{1},a_{2},\ldots,a_{k}) &=& \begin{pmatrix}
   a_{k} & -1 \\
    1    & 0 
    \end{pmatrix}\ldots\begin{pmatrix}
   a_{2} & -1 \\
    1    & 0 
   \end{pmatrix} \begin{pmatrix}
   a_{1} & -1 \\
    1    & 0 
   \end{pmatrix} \begin{pmatrix}
   a_{n} & -1 \\
    1    & 0 
   \end{pmatrix} \ldots \begin{pmatrix}
   a_{k+1} & -1 \\
    1    & 0 
   \end{pmatrix}\\
                                &=&  \begin{pmatrix}
   a_{k} & -1 \\
    1    & 0 
    \end{pmatrix}\ldots \begin{pmatrix}
   a_{1} & -1 \\
    1    & 0 
   \end{pmatrix} \begin{pmatrix}
   a_{n} & -1 \\
    1    & 0 
   \end{pmatrix} \ldots \begin{pmatrix}
   a_{k+1} & -1 \\
    1    & 0 
   \end{pmatrix}\begin{pmatrix}
   a_{k} & -1 \\
    1    & 0 
    \end{pmatrix} \\
	                              & & \ldots\begin{pmatrix}
   a_{1} & -1 \\
    1    & 0 
   \end{pmatrix}\begin{pmatrix}
   a_{1} & -1 \\
    1    & 0 
    \end{pmatrix}^{-1}\begin{pmatrix}
   a_{2} & -1 \\
    1    & 0 
   \end{pmatrix}^{-1} \ldots \begin{pmatrix}
   a_{k} & -1 \\
    1    & 0 
   \end{pmatrix}^{-1} \\
	                              &=& M_{k}(a_{1},\ldots,a_{k}) M_{n}(a_{1},\ldots,a_{n}) M_{k}(a_{1},\ldots,a_{k})^{-1}. \\
\end{eqnarray*}																

Soit $M$ un élément du groupe modulaire vérifiant $\mathcal{P}$. Soient $n \in \mathbb{N}^{*}$ et $(a_{1},\ldots,a_{n})$ un $n$-uplet d'entiers strictement positifs tels que $M_{n}(a_{1},\ldots,a_{n})=M$. On a 
\begin{eqnarray*}
M &=& M_{n}(a_{k+1},\ldots,a_{n},a_{1},a_{2},\ldots,a_{k})~({\rm car}~M~{\rm v\acute erifie}~\mathcal{P}) \\
  &=& M_{k}(a_{1},\ldots,a_{k}) M_{n}(a_{1},\ldots,a_{n}) M_{k}(a_{1},\ldots,a_{k})^{-1}~{\rm (par~le~calcul~pr\acute ec \acute edent)} \\
	&=& M_{k}(a_{1},\ldots,a_{k}) M M_{k}(a_{1},\ldots,a_{k})^{-1}.\\
\end{eqnarray*}
Donc, \[M M_{k}(a_{1},\ldots,a_{k})= M_{k}(a_{1},\ldots,a_{k}) M.\]

Soit $A$ une matrice du groupe modulaire. Il existe un entier strictement positif $k$ et un $k$-uplet d'entiers strictement positifs $(b_{1},\ldots,b_{k})$ tels que $A=M_{k}(b_{1},\ldots,b_{k})$. De même, il existe un entier strictement positif $l$ et un $l$-uplet d'entiers strictement positifs $(c_{1},\ldots,c_{l})$ tels que $A^{-1}=M_{l}(c_{1},\ldots,c_{l})$. On peut alors rajouter dans n’importe quelle solution de $M_{n}(a_{1},\ldots,a_{n})=M$ les $b_{i}$ et les $c_{j}$ de la façon suivante :
\[M_{n+k+l}(b_{1},\ldots,b_{k},c_{1},\ldots,c_{l},a_{1},\ldots,a_{n}) = M.\]
Par ce qui précède, $M$ commute avec $M_{k}(b_{1},\ldots,b_{k})=A$. Ainsi, $M$ commute avec n’importe quelle matrice du groupe modulaire, c'est-à-dire $M$ est dans le centre du groupe modulaire. Donc, $M=\pm Id$.

\end{proof}

\section{Résolution de $(E_{S})$}
\label{d}

Le but de cette section est de rechercher les $n$-uplets d'entiers strictement positifs solutions de $(E_{S})$ et de démontrer le théorème \ref{23}. Par la proposition \ref{21}, les solutions de cette équation ne sont pas invariantes par permutations circulaires.

\subsection{Construction récursive des solutions}

On considère les deux équations suivantes: \begin{equation}
\tag{$E_{S}^{1}$}
M_{n}(a_1,\ldots,a_n)=S,
\end{equation} et \begin{equation}
\tag{$E_{S}^{2}$}
M_{n}(a_1,\ldots,a_n)=-S.
\end{equation} L'opération (a) conserve les solutions des équations $(E_{S}^{1})$ et $(E_{S}^{2})$. L'opération (b) échange les solutions des équations $(E_{S}^{1})$ et $(E_{S}^{2})$. Avant d'étudier en détail ces équations on a besoin de deux résultats dûs à V.Ovsienko et S.Morier-Genoud (voir \cite{MO}) concernant les présentations minimales d'éléments de $PSL_{2}(\mathbb{Z})$. Si $M \in SL_{2}(\mathbb{Z})$ on note $\overline{M}$ la classe de $M$ dans $PSL_{2}(\mathbb{Z})$. On a vu dans l'introduction que l'écriture d'un élément de $SL_{2}(\mathbb{Z})$ sous la forme $M_{n}(a_{1},\ldots,a_{n})$ n'est pas unique mais on dispose cependant des deux résultats suivants:

\begin{thm}[\cite{MO},Théorème 6.3]
\label{51}

L'écriture $\overline{M}=\overline{M_{k}(c_{1},\ldots,c_{k})}$ avec des coefficients $c_{i}$ strictement positifs est unique si $k$ est le plus petit possible. Une telle écriture sera appelé la présentation minimale de $\overline{M}$.

\end{thm}

\begin{thm}[\cite{MO},proposition 6.4]
\label{52}

Si $\overline{M} \in PSL_{2}(\mathbb{Z})$ s'écrit sous la forme $\overline{M}=\overline{M_{k}(c_{1},\ldots,c_{k})}$, alors cette écriture est la présentation minimale de $\overline{M}$, si et seulement si $c_{i} \geq 2$, sauf peut-être aux extrémités de la séquence, c'est-à-dire, pour $c_{1}$ ou $c_{1}$ et $c_{2}$ et $c_{k}$, ou $c_{k-1}$ et $c_{k}$.

\end{thm}

On retourne maintenant à l'étude de $(E_{S}^{1})$ et $(E_{S}^{2})$. On commence par l'étude de ces deux équations pour les petites valeurs de $n$.

\begin{lem}
\label{53}

Si $n \in \{1,2,3,4\}$ les équations $(E_{S}^{1})$ et $(E_{S}^{2})$ n'ont pas de solution. Si $n=5$, $(E_{S}^{1})$ n'a pas de solution et $(E_{S}^{2})$ a une unique solution (1,1,2,1,1).

\end{lem}

\begin{proof}

On vérifie facilement que $M_{5}(1,1,2,1,1)=-S$. Donc, par le théorème \ref{52}, la présentation minimale de $S$ vu comme élément de $PSL_{2}(\mathbb{Z})$ est $\overline{S}=\overline{M(1,1,2,1,1)}$. Donc, d'après le théorème \ref{51}, $(E_{S})$ n'a pas de solution si $n \in \{1,2,3,4\}$ et la seule solution de $(E_{S})$ pour $n=5$ est (1,1,2,1,1). Donc, $(E_{S}^{1})$ et $(E_{S}^{2})$ n'ont pas de solution si $n \in \{1,2,3,4\}$. Si $n=5$, $(E_{S}^{1})$ n'a pas de solution et $(E_{S}^{2})$ a une unique solution (1,1,2,1,1).

\end{proof}

On va maintenant essayer de trouver tous les $n$-uplets d'entiers strictement positifs solutions de $(E_{S})$. Comme $S$ est d'ordre fini, une solution de $(E_{S})$ contient nécessairement un 1 (proposition \ref{41}) mais on dispose d'un résultat plus précis :
	
\begin{lem}
\label{54}

Si $(a_{1},\ldots,a_{n})$ est un $n$-uplet d'entiers strictement positifs solution de $(E_{S})$ alors il existe $i \in \llbracket 2~;~ n-1 \rrbracket$ tel que $a_{i}=1$.

\end{lem}

\begin{proof}

Si $(a_{1},\ldots,a_{n})$ est solution de $(E_{S})$. Il existe $\epsilon$ dans $\{-1,1\}$ tel que $M_{n}(a_{1},\ldots,a_{n})=\epsilon S$.  On a \begin{eqnarray*}
M_{n}(a_{1},\ldots,a_{n})=\epsilon S & \Longleftrightarrow & SM_{n}(a_{1},\ldots,a_{n})=-\epsilon Id \\
                         & \Longleftrightarrow & M_{n+1}(a_{1},\ldots,a_{n},0)=-\epsilon Id \\
												 & \Longleftrightarrow & M_{n+1}(a_{2},\ldots,a_{n},0,a_{1})=-\epsilon Id. \\
\end{eqnarray*}

\noindent Or, $\begin{pmatrix}
    a_{1} & -1 \\
    1    & 0 
    \end{pmatrix} \begin{pmatrix}
    0 & -1 \\
    1    & 0 
    \end{pmatrix} \begin{pmatrix}
    a_{n} & -1 \\
    1    & 0 
    \end{pmatrix} = \begin{pmatrix}
    -a_{1}-a_{n} & 1 \\
    -1    & 0 
    \end{pmatrix}= -\begin{pmatrix}
    a_{1}+a_{n} & -1 \\
    1    & 0 
    \end{pmatrix}.$
\\
\\ \noindent Donc, $M_{n}(a_{1},\ldots,a_{n})=\epsilon S \Longleftrightarrow  M_{n-1}(a_{2},\ldots,a_{n}+a_{1})=\epsilon Id$. 
\\
\\ \noindent Donc, la séquence $(a_{2},\ldots,a_{n}+a_{1})$ contient un $1$ d'après le lemme \ref{35}. Comme $a_{n}+a_{1} \geq 2$, il existe $i$ dans $\llbracket 2~;~ n-1 \rrbracket$ tel que $a_{i}=1$.

\end{proof}

\noindent On dispose également du résultat ci-dessous :

\begin{prop}
\label{541}

$(a_{1},\ldots,a_{n})$ est solution de $(E_{S})$ si et seulement si $(a_{n},\ldots,a_{1})$ est solution de $(E_{S})$.

\end{prop}

\begin{proof} La preuve suivante est une adaptation de la remarque 2.6 de \cite{CH}.
\\
\\On pose $K=\begin{pmatrix}
    0 & 1 \\
    1    & 0 
    \end{pmatrix}$. Supposons $(a_{1},\ldots,a_{n})$ solution de $(E_{S})$. Il existe $\epsilon$ dans $\{\pm 1\}$ tel que \[M_{n}(a_{1},\ldots,a_{n})=\epsilon S.\] De plus,  
		\[\begin{pmatrix}
   a & -1 \\
   1 &  0 
    \end{pmatrix}^{-1}=K\begin{pmatrix}
    a & -1 \\
    1 &  0 
    \end{pmatrix}K~{\rm et}~K^{2}=Id.\] Donc,
		
\begin{eqnarray*}
M_{n}(a_{n},\ldots,a_{1}) &=& (K\begin{pmatrix}
    a_{1} & -1 \\
      1   &  0 
    \end{pmatrix}^{-1}K)\ldots(K\begin{pmatrix}
    a_{n} & -1 \\
      1   & 0 
    \end{pmatrix}^{-1}K)\\ 
		                   &=& K(\begin{pmatrix}
    a_{n} & -1 \\
     1    & 0 
    \end{pmatrix}\ldots\begin{pmatrix}
    a_{1} & -1 \\
     1    & 0 
    \end{pmatrix})^{-1}K\\
		              &=& K M_{n}(a_{1},\ldots,a_{n})^{-1}K\\
									&=& -\epsilon K S K~{\rm car}~(a_{1},\ldots,a_{n})~{\rm solution~de}~(E_{S}).\\
									&=& \epsilon S~{\rm car}~KS=-SK.\\								
\end{eqnarray*}

\noindent Donc, $(a_{n},\ldots,a_{1})$ est solution de $(E_{S})$. 
\\
\\Si $(a_{n},\ldots,a_{1})$ est solution de $(E_{S})$ alors par ce qui précède $(a_{1},\ldots,a_{n})$ est solution de $(E_{S})$.

\end{proof}	

\begin{prop}
\label{55}

Tout $n$-uplet d'entiers strictement positifs solution de $(E_{S}^{1})$ (resp. $(E_{S}^{2})$) peut être obtenu en appliquant des opérations (a) et un nombre impair d'opérations (b) (resp. pair) à (1,1,2,1,1).

\end{prop}
 
\begin{proof}

Soit $(a_{1},\ldots,a_{n})$ un $n$-uplet d'entiers strictement positifs solution de $(E_{S})$. D'après la proposition précédente, il existe $j$ dans $\llbracket 2~;~ n-1 \rrbracket$ tel que $a_{j}=1$. Montrons que l'on peut toujours voir $(a_{1},\ldots,a_{n})$ comme l'image par l'opération (a) ou (b) d'un $k$-uplet d'entiers strictement positifs (qui sera toujours solution de $(E_{S})$) sauf si $(a_{1},\ldots,a_{n})=(1,1,2,1,1)$.
\\
\\ A) Si $j \in \llbracket 3~;~ n-2 \rrbracket$ alors $n \geq 6$. Si $a_{j-1}>1$ et $a_{j+1}>1$ alors $(a_{1},\ldots,a_{n})$ est l'image par l'opération (a) du $(n-1)$-uplet $(a_{1},\ldots,a_{j-1}-1,a_{j+1}-1,\ldots,a_{n})$. Si $a_{j-1}=1$ alors $(a_{1},\ldots,a_{n})$ est l'image par l'opération (b) du $(n-3)$-uplet $(a_{1},\ldots,a_{j-2}+a_{j+1}-1,\ldots,a_{n})$. De même si $a_{j+1}=1$.
\\
\\ B) Si pour tout $i$ appartenant à $\llbracket 3~;~ n-2 \rrbracket$ $a_{i}>1$. On a alors $a_{2}=1$ ou $a_{n-1}=1$. On suppose que $a_{2}=1$ alors:
\\ 
\\Soit $a_{1}>1$ et dans ce cas $(a_{1},\ldots,a_{n})$ est l'image par l'opération (a) de $(a_{1}-1,a_{3}-1,\ldots,a_{n})$.
\\
\\ Soit $a_{1}=1$ et dans ce cas on regarde $a_{n-1}$. Si $a_{n-1}>1$ alors, d'après le théorème \ref{52}, $(a_{1},\ldots,a_{n})$ est la présentation minimale de $S$ vu comme élément de $PSL_{2}(\mathbb{Z})$ c'est-à-dire $(a_{1},\ldots,a_{n})=(1,1,2,1,1)$. Si $a_{n-1}=1$ alors soit $a_{n}>1$ et dans ce cas $(a_{1},\ldots,a_{n})$ est l'image par l'opération (a) du $(n-1)$-uplet $(a_{1},\ldots,a_{n-2}-1,a_{n}-1)$, soit $a_{n}=1$ et donc, d'après le théorème \ref{52}, $(a_{1},\ldots,a_{n})$ est la présentation minimale de $S$ vu comme élément de $PSL_{2}(\mathbb{Z})$ c'est-à-dire $(a_{1},\ldots,a_{n})=(1,1,2,1,1)$.
\\
\\La démonstration est la même dans le cas $a_{n-1}=1$.

\end{proof}

\subsection{Démonstration du théorème \ref{23}}
 On a \[M_{n}(a_{1},\ldots,a_{n})=S~({\rm resp.}~-S) \Longleftrightarrow M_{n-1}(a_{1}+a_{n},\ldots,a_{n-1})=Id~({\rm resp.}~-Id).\]

i) Si $(a_{1},\ldots,a_{n})$ est solution de $(E_{S})$ alors $M_{n-1}(a_{1}+a_{n},\ldots,a_{n-1})=\pm Id$. Par le théorème \ref{34}, $(a_{1}+a_{n},\ldots,a_{n-1})$ est la quiddité associée à une 3$d$-dissection d'un polygone convexe $R$ à $n-1$ sommets.
\\
\\ Il existe un $j$ tel que le sommet $1$ soit relié au sommet $j$ et tel que le sous-polygone formé par les sommets $1,\ldots, j$ contienne exactement $a_{1}$ sous-polygones utilisant le sommet $1$. On note $R_{1}$ le sous-polygone formé par les sommets $1,\ldots,j$ et $R_{2}$ le sous-polygone formé par les sommets $j,\ldots,n-1,1$ qui contient lui $a_{n}$ sous-polygones utilisant le sommet 1. $R_{1}$ et $R_{2}$ sont eux même potentiellement décomposés en sous-polygones. Si on insère un quadrilatère entre $R_{1}$ et $R_{2}$ on obtient une 3$d$-dissection échancrée dont la quiddité est $(a_{1},\ldots,a_{n})$.   
\\
\\ \indent ii) Soient $n \geq 5$ et $(a_{1},\ldots,a_{n})$ la quiddité d'une 3$d$-dissection échancrée d'un polygone convexe $P$ à $n+1$ sommets.
\\
\\La décomposition de $P$ contient un unique quadrilatère dont les sommets sont numérotés $0$, $1$, $n$ et un dernier sommet numéroté $j$ avec $j \in \llbracket 3~;~ n-2 \rrbracket$. Ce quadrilatère permet de définir deux sous-polygones: le polygone utilisant les sommets $1, 2, \ldots, j$ que l'on notera $P_{1}$ et le polygone utilisant les sommets $j, j+1, \ldots, n$ que l'on notera $P_{2}$. $P_{1}$ et $P_{2}$ peuvent eux-même être décomposés en sous-polygones. Notons que $P_{1}$ et $P_{2}$ n'ont que le sommet $j$ en commun. On supprime le quadrilatère entre $P_{1}$ et $P_{2}$ et on recolle ces deux polygones en identifiant la diagonale reliant 1 à $j$ et celle reliant $j$ à $n$. On obtient une 3$d$-dissection dont la quiddité est $(a_{1}+a_{n},a_{2},\ldots,a_{n-1})$ (car les sous-polygones qui utilisaient le sommet $n$ dans la décomposition de $P$ utilisent maintenant le sommet $1$). Par le théorème \ref{34}, on a $M_{n-1}(a_{1}+a_{n},\ldots,a_{n-1})=\pm Id$  et donc $M_{n}(a_{1},\ldots,a_{n})=\pm S$. 
\\
\\Donc, $(a_{1},\ldots,a_{n})$ est solution de $(E_{S})$. 

\qed

\section{Résolution de $(E_{T})$}
\label{e}

Le but de cette section est de rechercher les $n$-uplets d'entiers strictement positifs solutions de $(E_{T})$ et de démontrer le théorème \ref{25}. Par la proposition \ref{21}, les solutions de cette équation ne sont pas invariantes par permutations circulaires. De plus, la proposition \ref{541} ne se généralise pas pour $(E_{T})$. Par exemple, on a $-T=M_{3}(1,1,2)$ et $M_{3}(2,1,1)=\begin{pmatrix}
     -1 & 0 \\
      1   & -1 
    \end{pmatrix} \neq \pm T$.

\subsection{Construction récursive des solutions}

On considère les deux équations suivantes: \begin{equation}
\tag{$E_{T}^{1}$}
M_{n}(a_1,\ldots,a_n)=T,
\end{equation} et \begin{equation}
\tag{$E_{T}^{2}$}
M_{n}(a_1,\ldots,a_n)=-T.
\end{equation} L'opération (a) conserve les solutions des équations $(E_{T}^{1})$ et $(E_{T}^{2})$. L'opération (b) échange les solutions des équations $(E_{T}^{1})$ et $(E_{T}^{2})$. On commence par l'étude de ces deux équations pour les petites valeurs de $n$.

\begin{lem}
\label{61}

i)Si $n \in \{1,2\}$ les équations $(E_{T}^{1})$ et $(E_{T}^{2})$ n'ont pas de solution. 
\\ii)Si $n=3$, $(E_{T}^{1})$ n'a pas de solution et $(E_{T}^{2})$ a une unique solution (1,1,2).
\\iii)Si $n=4$, $(E_{T}^{1})$ n'a pas de solution et $(E_{T}^{2})$ a pour solution (1,2,1,3) et (2,1,2,2). 

\end{lem}

\begin{proof}

On a $M_{3}(1,1,2)=-T$. Donc, par le théorème \ref{52}, la présentation minimale de $T$ vu comme élément de $PSL_{2}(\mathbb{Z})$ est $\overline{T}=\overline{M(1,1,2)}$. Donc, d'après le théorème \ref{51}, $(E_{T})$ n'a pas de solution si $n \in \{1,2\}$ et la seule solution de $(E_{T})$ pour $n=3$ est (1,1,2).
\\
\\ Soit $(a,b,c,d)$ une solution de $(E_{T})$. On a $M_{4}(a,b,c,d)=\begin{pmatrix}
    abcd-ad-ab-cd+1 & -bcd+b+d \\
    abc-a-c         & 1-bc 
    \end{pmatrix}.$ On a nécessairement $bc=2$. On en déduit que $a=c$ et $d=b+1$. On a deux cas : $a=c=1$, $b=2$ et $d=3$ ou $a=c=2$, $b=1$ et $d=2$. On vérifie que ces deux 4-uplets sont des solutions de $(E_{T}^{2})$.

\end{proof}

On va maintenant essayer de trouver tous les $n$-uplets d'entiers strictement positifs solutions de $(E_{T})$. On donne ci-dessous deux lemmes utiles pour cette étude :
	
\begin{lem}
\label{62}

$(a_{1},\ldots,a_{n})$ est un $n$-uplet d'entiers strictement positifs solution de $(E_{T})$ si et seulement si $(a_{1},\ldots,a_{n}-1)$ est solution de $(E_{Id})$.

\end{lem}

\begin{proof}

\begin{eqnarray*}
(a_{1},\ldots,a_{n})~{\rm est~solution~de}~(E_{T}). & \Longleftrightarrow & {\rm Il~existe}~\epsilon~{\rm dans}~\{-1,1\}~{\rm tel~que}~M_{n}(a_{1},\ldots,a_{n})=\epsilon T. \\ 
                                                    & \Longleftrightarrow & {\rm Il~existe}~\epsilon~{\rm dans}~\{-1,1\}~{\rm tel~que}~ T^{-1}M_{n}(a_{1},\ldots,a_{n})=\epsilon Id. \\
                                                    & \Longleftrightarrow & {\rm Il~existe}~\epsilon~{\rm dans}~\{-1,1\}~{\rm tel~que}~ M_{n}(a_{1},\ldots,a_{n}-1)=\epsilon Id. \\
\end{eqnarray*}

\end{proof}

\begin{lem}
\label{63}

$(a_{1},\ldots,a_{n-1},1)$ est un $n$-uplet d'entiers strictement positifs solution de $(E_{T})$ si et seulement si $(a_{1},\ldots,a_{n-1})$ est solution de $(E_{S})$.

\end{lem}

\begin{proof}

\begin{eqnarray*}
(a_{1},\ldots,a_{n-1},1)~{\rm solution~de}~(E_{T}). & \Longleftrightarrow & {\rm Il~existe}~\epsilon~{\rm dans}~\{\pm 1\}~{\rm tel~que}~M_{n}(a_{1},\ldots,a_{n-1},1)=\epsilon T. \\ 
                                                    & \Longleftrightarrow & {\rm Il~existe}~\epsilon~{\rm dans}~\{\pm 1\}~{\rm tel~que}~ T^{-1}M_{n}(a_{1},\ldots,a_{n-1},1)=\epsilon Id. \\
                                                    & \Longleftrightarrow & {\rm Il~existe}~\epsilon~{\rm dans}~\{\pm 1\}~{\rm tel~que}~ M_{n}(a_{1},\ldots,a_{n-1},0)=\epsilon Id. \\
																									  & \Longleftrightarrow & {\rm Il~existe}~\epsilon~{\rm dans}~\{\pm 1\}~{\rm tel~que}~ M_{n}(a_{1},\ldots,a_{n-1})=-\epsilon S. \\
\end{eqnarray*}

\end{proof}

On peut maintenant démontrer le résultat suivant :

\begin{prop}
\label{64}

Tout $n$-uplet d'entiers strictement positifs solution de $(E_{T}^{1})$ (resp. $(E_{T}^{2})$) peut être obtenu en appliquant des opérations (a) et un nombre impair d'opérations (b) (resp. pair) à (1,1,2).

\end{prop}
 
\begin{proof}

Soient $n \geq 4$ et $(a_{1},\ldots,a_{n})$ un $n$-uplet d'entiers strictement positifs solution de $(E_{T})$. Montrons que l'on peut toujours voir $(a_{1},\ldots,a_{n})$ comme l'image par l'opération (a) ou (b) d'un $k$-uplet d'entiers strictement positifs (qui sera toujours solution de $(E_{T})$).
\\
\\Si $n=4$, les solutions de $(E_{T})$ sont (1,2,1,3) et (2,1,2,2) (voir lemme \ref{61}). Ces deux 4-uplets peuvent être obtenus en appliquant l'opération (a) à (1,1,2). On suppose maintenant $n \geq 5$. On distingue deux cas :
\\
\\A) Si $a_{n}=1$. D'après le lemme \ref{63}, $(a_{1},\ldots,a_{n-1})$ est solution de $(E_{S})$. D'après la proposition \ref{55}, $(a_{1},\ldots,a_{n-1})$ peut être obtenu en appliquant des opérations (a) et (b) à (1,1,2,1,1). On en déduit que $(a_{1},\ldots,a_{n-1},1)$ peut être obtenu en appliquant des opérations (a) et (b) à (1,1,2,1,1,1). Or, (1,1,2,1,1,1) est l'image par l'opération (b) de (1,1,2).
\\
\\B) Si $a_{n} \geq 2$. S'il existe  $j \in \llbracket 3~;~ n-2 \rrbracket$ tel que $a_{j}=1$ alors on peut procéder comme de la même façon que le cas A) de la proposition \ref{55}.
\\
\\Si pour tout $i$ dans $\llbracket 3~;~ n-2 \rrbracket$ $a_{i}>1$. Par le lemme \ref{62}, $(a_{1},\ldots,a_{n}-1)$ est solution de $(E_{Id})$. Par le théorème \ref{34}, $(a_{1},\ldots,a_{n}-1)$ est la quiddité d'une 3$d$-dissection d'un polygone convexe $P$ à $n$ sommets. Cette décomposition est constitué d'un seul polygone (et dans ce cas $n$ est un multiple de 3) ou bien possède au moins deux sous-polygones extérieurs. Dans les deux cas, $(a_{1},\ldots,a_{n}-1)$ contient au moins deux 1. En particulier, il existe $j$ dans $\{1,2,n-1\}$ tel que $a_{j}=1$. \\
\\Si $a_{n-1}=1$ alors $(a_{1},\ldots,a_{n})$ est l'image par l'opération (a) du $(n-1)$-uplet $(a_{1},\ldots,a_{n-2}-1,a_{n}-1)$. On suppose maintenant $a_{n-1}>1$. On a trois cas :

\begin{itemize}

\item Si $a_{2}=1$ et $a_{1} >1$ alors $(a_{1},\ldots,a_{n})$ est l'image par l'opération (a) du $(n-1)$-uplet $(a_{1}-1,a_{3}-1,a_{4},\ldots,a_{n})$. 

\item Si $a_{2}=1$ et $a_{1}=1$ alors, par le théorème \ref{52}, $(a_{1},\ldots,a_{n})$ est la présentation minimale de $T$ et donc $n=3$. Ceci est impossible puisque $n \geq 5$.

\item Si $a_{2}>1$ alors, par le théorème \ref{52}, $(a_{1},\ldots,a_{n})$ est la présentation minimale de $T$ et donc $n=3$. Ceci est impossible puisque $n \geq 5$.

\end{itemize}

\end{proof}

\begin{rem}

{\rm Ce théorème permet de voir que les solutions de $(E_{T})$ contiennent toujours un 1. En particulier, la proposition 4.1 n'a pas de réciproque.}

\end{rem}

\subsection{Démonstration du théorème \ref{25}}

i) Cela découle du lemme \ref{63} et du théorème \ref{23}. 
\\
\\ \noindent ii) Soit $(a_{1},\ldots,a_{n})$ une solution de $(E_{T})$ avec $a_{n} \geq 2$. Par le lemme \ref{62}, $(a_{1},\ldots,a_{n}-1)$ est un $n$-uplet d'entiers strictement positifs solution de $(E_{Id})$. Par le théorème \ref{34}, celui-ci est la quiddité d'une 3$d$-dissection d'un polygone convexe $P$ à $n$ sommets. On rajoute un quadrilatère coupé en un triangle de poids -1 et un triangle de poids 1 sur le côté reliant le sommet 1 de $P$ au sommet $n$ de $P$ de telle sorte que les deux triangles utilisent le sommet 1 de $P$. 

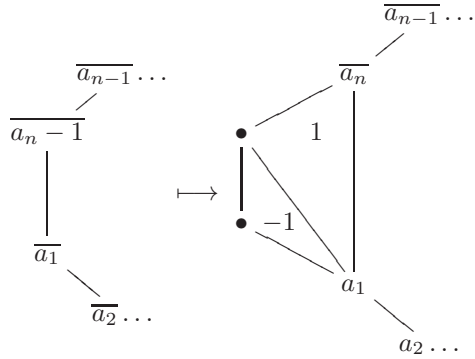
$$
\shorthandoff{; :!?}
\qquad
\xymatrix @!0 @R=0.40cm @C=0.50cm
{
&&
\\
&&
\\
&&\overline{a_{n-1}}\ldots\ar@{-}[lldd]&
\\
&&& 
\\
\overline{a_{n}-1}\ar@{-}[dddd]&& 
\\
&&&
\\
&&&& \longmapsto
\\
&&& 
\\
\overline{a_{1}} 
\\
&& 
\\
&&\overline{a_{2}}\ldots\ar@{-}[lluu]
}  
\xymatrix @!0 @R=0.40cm @C=0.50cm
{
&&&&&\overline{a_{n-1}}\ldots\ar@{-}[lldd]&
\\
&&&&&
\\
&&&\overline{a_{n}}\ar@{-}[ddddddd]&& 
\\
&&&&&
\\
\bullet\ar@{-}[ddd]\ar@{-}[rrruu] &&1
\\
&
\\
& 
\\
\bullet \ar@{-}[rrrdd] &-1
\\
&&&&&
\\
&&&a_{1}\ar@{-}[llluuuuu] 
\\
&&&&
\\
&&&&&a_{2}\ldots\ar@{-}[lluu]
}
$$
On obtient une 3$d$-dissection coiffée d'un polygone convexe à $n+2$ sommets de quiddités $(a_{1},\ldots,a_{n})$.
\\
\\ Soit $(a_{1},\ldots,a_{n})$ la quiddité d'une 3$d$-dissection coiffée d'un polygone convexe à $n+2$ sommets $Q$. En particulier, $a_{n} \geq 2$. On retire le triangle de poids -1 et le triangle de poids 1 adjacent à celui-ci. On obtient une 3$d$-dissection d'un polygone convexe à $n$ sommets de quiddité $(a_{1},\ldots,a_{n}-1)$. Par le théorème \ref{34}, $(a_{1},\ldots,a_{n}-1)$ est une solution de $(E_{Id})$. Donc, par le lemme \ref{62}, $(a_{1},\ldots,a_{n})$ est une solution de $(E_{T})$.

\qed

\noindent \textbf{Remerciement.} Ce travail a été financé par la région Grand Est et l'Université de Reims Champagne-Ardenne.

\end{document}